\documentclass{amsart}
\usepackage{latexsym, amssymb, amsmath, qtree}

\newtheorem{theorem}{Theorem}[section]
\newtheorem{lemma}[theorem]{Lemma}
\newtheorem{corollary}[theorem]{Corollary}

\numberwithin{equation}{section}
\DeclareMathOperator{\clf}{\phi}
\newcommand{\sob}[1]{\beta(#1)}
\newcommand{\sbf}[1]{(-1)^{\beta(#1)}}
\makeindex
\begin{document}

\title{The Clifford Twist}

\author{John W. Bales}
\date{}
\address{Department of Mathematics\\
         Tuskegee University\\
	 Tuskegee, AL 36088\\
	 USA}
\email{jbales@mytu.tuskegee.edu}


\begin{abstract}
  This is an elementary exposition of the twisted group algebra representation of simple Clifford algebras.
\end{abstract}
\maketitle

\section{Clifford Algebra}

Clifford Algebra is an algebra defined on a potentially infinite set $e_1,e_2,e_3,\cdots$ of linearly independent unit vectors, their finite
products (called multi-vectors) and the unit scalar 1 (denoted $e_0$). Every element of the algebra is a linear combination of these basis
elements over some ring, usually the real numbers.

The vectors $e_1,e_2,e_3,\cdots$ are referred to as `1-blades.' A product of two vectors is called a `2-blade.' three vectors a `3-blade' and so
forth. The scalar $e_0$ is a `0-blade.' An $n$-blade multi-vector is said to be of grade $n.$

There are four fundamental multiplication properties of 1-blades.
\begin{enumerate}
\item The square of 1-blades is $\mu$ (where $\mu^2=1$). 
\item The product of 1-blades is anti-commutative.
\item The product of 1-blades is associative.
\item Every $n$-blade can be factored into the product of $n$ distinct 1-blades.
\end{enumerate}

The product of $e_i$ and $e_j$ is denoted $e_{ij}$ if $i<j$ and by $-e_{ij}$ if $i>j.$ Likewise for higher order blades. For example, if
$i<j<k$ then $e_ie_je_k=e_{ijk}.$ 

Any two $n$-blades may be multiplied by first factoring them into 1-blades. For example, the product of $e_{134}$ and $e_{23},$ is computed as follows:
\begin{align*} e_{134}e_{23}&=e_1e_3e_4e_2e_3\\
                            &=-e_1e_4e_3e_2e_3\\
			    &=e_1e_4e_2e_3e_3\\
			    &=\mu e_1e_4e_2\\
			    &=-\mu e_1e_2e_4\\
			    &=-\mu e_{124}
\end{align*}

\section{Representing Clifford Algebra as a twisted group algebra}

Each of the basis elements of Clifford algebra $1,e_1, e_2,e_{12},e_3,\cdots$ can be associated with an element of the set $G$ of
non-negative integers.

Each vector $e_k$ is associated with the integer $2^{k-1}$ and the scalar $e_0$ is associated with 0. A multi-vector is associated with
the sum of the integers associated with its vector factors. Thus, for example, the multi-vector $e_{134}$ is associated with the sum $2^0+2^2+2^3=13.$
Notice that the binary representation of 13 is 1101 with bits 1, 3 and 4 set. We will represent the sequence $1,e_1, e_2,e_{12},e_3,\cdots$ by the
sequence $i_0,i_1,i_2,i_3,i_4,\cdots$ where the subscript of $i$ is the number associated
with the corresponding vector or multi-vector.

Notice that, since the square of a vector $e_k$ is $\mu$ which is either 1 or $-1$, the product of two basis elements $i_p$ and $i_q$ will always be
either $i_r$ or $-i_r$ where $r$ is the XOR (exclusive or) of the binary representations of integers $p$ and $q$. The set $G$ of
non-negative integers is a group under XOR. For brevity, we will denote the operation $p$ XOR $q$ by simple concatenation $pq.$ Thus there is a
function $\clf$ mapping $G\times G$ into $\{-1,1\}$ such that if $p,q\in G$ then

\begin{equation}
  i_pi_q=\clf(p,q)i_{pq}
\end{equation}

thereby representing Clifford algebra as a twisted group algebra.\\

Let $2p$ denote the double of $p.$ Notice that the vector factors of $i_{2p}$ are the \emph{successors} of the vector factors of $i_p$ in the sense
that $e_k$ is a vector factor of $i_{2p}$ if and only if $e_{k-1}$ is a vector factor of $i_p.$ For example, $i_{13}=e_{134}$ and $i_{26}=e_{245}.$
This is more intuitive if the subscripts are represented in binary. $13=1101_B$ with bits 1,3 and 4 set, and $2(13)=26=11010_B$ with bits 2,4 and 5
set. Multiplying by 2 in binary shifts bits to the left and appends a 0 on the right.

The next two lemmas are then immediately obvious.

\begin{lemma}
  \item $e_1i_{2p}=i_{2p+1}$
\end{lemma}

\begin{lemma}
  \item $e_1i_{2p+1}=\mu i_{2p}$
\end{lemma}

Let $\sob{p}$ denote the sum of the bits of $p.$ Then $\beta(p)$ is the grade of $i_p.$ The remaining lemmas follow from the fact that $i_{2p}$
contains exactly $\sob{p}$ vector factors and
$e_1$ must be `commuted' with each of them to `find its place' so to speak.

\begin{lemma}
  \item $i_{2p}e_1=\sbf{p}i_{2p+1}$
\end{lemma}
\begin{lemma}
  \item $i_{2p+1}e_1=\sbf{p}\mu i_{2p}$
\end{lemma}

\begin{theorem}\label{T:sign2}
There is a twist $\clf(p,q)$ mapping $G\times G$ into
$\{-1,1\}$ such that if $p,\,q\in G,$ then
$i_pi_q=\clf(p,q)i_{pq}.$
\end{theorem}
\begin{proof}
Let $G_n=\{p\:|\:0\le p < 2^n\}$ with group operation ``bit-wise exclusive or.''

To begin with, $i_0i_0=\clf(0,0)i_0=1$ provided $\clf(0,0)=1.$

This defines the twist for $G_0.$

If $p$ and $q$ are in $G_{n+1},$ then there are elements $u$ and $v$ in $G_n$ such that one of the following is true:
\begin{enumerate}
 \item $p=2u$ and $q=2v$
 \item $p=2u$ and $q=2v+1$
 \item $p=2u+1$ and $q=2v$
 \item $p=2u+1$ and $q=2v+1$
\end{enumerate}
Assume $\clf$ is defined for $u,v\in G_n,$ then consider these four cases in order.

\begin{enumerate}
 \item $p=2u$ and $q=2v$\\
       \begin{align*}i_pi_q &= i_{2u}i_{2v}\\
                            &= \clf(u,v)i_{2uv}\\
                            &= \clf(2u,2v)i_{(2u)(2v)}\\
			    &= \clf(p,q)i_{pq}
       \end{align*} provided $\clf(2u,2v)=\clf(u,v).$
 \item $p=2u$ and $q=2v+1$
       \begin{align*}i_pi_q & =i_{2u}i_{2v+1}\\
                            & =i_{2u}e_1i_{2v}\\
                            & =\sbf{u}e_1i_{2u}i_{2v}\\
                            & =\sbf{u}e_1\clf(2u,2v)i_{2uv}\\
                            & =\sbf{u}\clf(u,v)i_{2uv+1}\\
 			    & =\clf(2u,2v+1)i_{2uv+1}\\
			    & =\clf(p,q)i_{pq}
       \end{align*} provided $\clf(2u,2v+1)=\sbf{u}\clf(u,v).$
 \item $p=2u+1$ and $q=2v$
       \begin{align*} i_pi_q &=i_{2u+1}i_{2v}\\
                             &=e_1i_{2u}i_{2v}\\
			     &=e_1\clf(u,v)i_{2uv}\\
			     &=\clf(u,v)i_{2uv+1}\\
			     &=\clf(2u+1,v)i_{2uv+1}\\
			     &=\clf(p,q)i_{pq}      
       \end{align*} provided $\clf(2u+1,2v)=\clf(u,v).$
 \item $p=2u+1$ and $q=2v+1$
       \begin{align*} i_pi_q &=i_{2u+1}i_{2v+1}\\
                             &=e_1i_{2u}e_1i_{2v}\\
                             &=\sbf{u}e_1e_1i_{2u}i_{2v}\\
			     &=\sbf{u}\mu\clf(u,v)i_{2uv}\\
			     &=\clf(2u+1,2v+1)i_{2uv}\\
			     &=\clf(p,q)i_{pq}
       \end{align*} provided $\clf(2u+1,2v+1)=\sbf{u}\mu\clf(u,v).$
\end{enumerate}
\end{proof}

\begin{corollary}
 Assume $p,q\in G_n.$ The Clifford algebra twist can be defined recursively as follows:
 \begin{enumerate}
  \item $\clf(0,0)=1$
  \item $\clf(2p,2q)=\clf(2p+1,2q)=\clf(p,q)$
  \item $\clf(2p,2q+1)=\sbf{p}\clf(p,q)$
  \item $\clf(2p+1,2q+1)=\sbf{p}\mu\clf(p,q)$
 \end{enumerate}
\end{corollary}

 Stated another way
 
 \begin{equation}
   \begin{bmatrix} \clf(2p,2q)   & \clf(2p,2q+1)\\ \clf(2p+1,2q) & \clf(2p+1,2q+1) \end{bmatrix}
   = \clf(p,q)\begin{bmatrix} 1 & \sbf{p}\\ 1 & \sbf{p}\mu \end{bmatrix}
 \end{equation}

\section{Recursive generation of twist matrices for higher dimensions}

 The twist matrix for dimension one is found when $p=q=0$
  
 \[\left[ \begin{array}{rr}
          1 & 1 \\1 & \mu
  \end{array}\right]\]

For two dimensions, $0\le p\le1,0\le q\le 1,$ the twist matrix is

\[\left[ \begin{array}{rrrr}
            1 &    1 &   1 &   1\\
            1 &  \mu &   1 &  \mu\\
            1 &   -1 & \mu & -\mu\\
            1 & -\mu & \mu &   -1
         \end{array}\right]\]

 For $\mu=-1$ these coincide with the twist tables for complex numbers and quaternions. For dimension 3, however, we do not get the twist table for
the octonions, rather

 \[\left[ \begin{array}{rrrrrrrr}
            1 &    1 &   1 &   1 &   1 &    1&    1&    1\\
            1 &  \mu &   1 &  \mu & 1 &  \mu &   1 &  \mu\\
            1 &   -1 & \mu & -\mu & 1 &   -1 & \mu & -\mu \\
            1 & -\mu & \mu &   -1 & 1 & -\mu & \mu &   -1 \\
            1 & -1 & -1 & 1 & \mu & -\mu & -\mu & \mu \\
            1 & -\mu & -1 & \mu & \mu & -1 & -\mu & 1\\
            1 & 1 & -\mu & -\mu & \mu & \mu & -1 & -1\\
            1 & \mu & -\mu & -1 & \mu & 1 & -1 & -\mu 
         \end{array}\right]\]

For dimension four the twist matrix is too large to represent in this form, so we make the following substitutions:

\begin{equation}
  A = \left[\begin{array}{rr}1 & 1 \\1 & \mu \end{array}\right]
\end{equation}

\begin{equation}
  B = \left[\begin{array}{rr}1 & -1 \\1 & -\mu\end{array}\right]
\end{equation}

The matrices $A$ and $B$ are simply the values of $M(p)=\begin{bmatrix} 1 & \sbf{p}\\ 1 & \sbf{p}\mu \end{bmatrix}$ when $\sbf{p}$ is positive and
negative, respectively.

Then the dimension 4 twist table can be represented as follows.

 \[\left[ \begin{array}{rrrrrrrr}
            A &    A &   A &   A &   A &    A &    A & A \\
            B &  \mu B &   B &  \mu B & B &  \mu B &   B &  \mu B\\
            B &   -B & \mu B & -\mu B & B &   -B & \mu B & -\mu B \\
            A & -\mu A & \mu A &   -A & A & -\mu A & \mu A &   -A \\
            B & -B & -B & B & \mu B & -\mu B & -\mu B & \mu B \\
            A & -\mu A & -A & \mu A & \mu A & -A & -\mu A & A\\
             A & A & -\mu A & -\mu A & \mu A & \mu A & -A & -A\\
             B & \mu B & -\mu B & -B & \mu B & B & -B & -\mu B
          \end{array}\right]\]

The twist tables for the various dimensions can be generated recursively beginning with $A$ for dimension 1, then making the following replacements to
generate the twist table for each successively higher dimension:

\begin{equation}
  A \Longrightarrow \left[\begin{array}{rr}A & A \\B & \mu B\end{array}\right]
\end{equation}

\begin{equation}
  B \Longrightarrow \left[\begin{array}{rr}B & -B \\A & -\mu A\end{array}\right]
\end{equation}

\section{A tree for computing the Clifford twist}

In \cite{B2009} a tree for computing the Cayley-Dickson twist is described. The same procedure applies to the Clifford twist.

The tree consists of only four components which repeat indefinitely, beginning at node $A$. There are two versions, one for each value of $\mu.$

\begin{figure}[ht]
\Tree [.A [. A A ] [. B B ] ]
\Tree [.--A [. --A --A ] [--B --B ] ]
\Tree [.B [. B --B ] [. A --A ] ]
\Tree [.--B [. --B B ] [. --A A ] ]
  \caption{Clifford twist tree for $\mu=1.$}
\end{figure}

\begin{figure}[ht]
\Tree [.A [. A A ] [. B --B ] ]
\Tree [.--A [. --A --A ] [--B B ] ]
\Tree [.B [. B --B ] [. A A ] ]
\Tree [.--B [. --B B ] [. --A --A ] ]
  \caption{Clifford twist tree for $\mu=-1.$}
\end{figure}

Let us illustrate the use of the tree to compute the product $i_{2636}i_{1143}$ given $\mu=-1.$

\begin{enumerate}
  \item Convert the subscripts to binary notation. $2636=101001001100_B$ and $1143=10001110111_B.$
  \item Pair the bits of the first subscript with the bits of the second by placing one over the other. Pad the smaller with zero bits if necessary.
\\
     $\xrightarrow[0]{1},
     \xrightarrow[1]{0},
     \xrightarrow[0]{1},
     \xrightarrow[0]{0},
     \xrightarrow[0]{0},
     \xrightarrow[1]{1},
     \xrightarrow[1]{0},
     \xrightarrow[1]{0},
     \xrightarrow[0]{1},
     \xrightarrow[1]{1},
     \xrightarrow[1]{0},
     \xrightarrow[1]{0}$  

  \item Each binary pair is an instruction for traversing one of the four tree components. A zero is an instruction to move down a left branch and a
one is an instruction to move down a right branch. The result is the following path.\\

\begin{eqnarray*}
   A &\xrightarrow[0]{1}&B\\
     &\xrightarrow[1]{0}&-B\\
     &\xrightarrow[0]{1}&-A\\
     &\xrightarrow[0]{0}&-A\\
     &\xrightarrow[0]{0}&-A\\
     &\xrightarrow[1]{1}&B\\
     &\xrightarrow[1]{0}&-B\\
     &\xrightarrow[1]{0}&B\\
     &\xrightarrow[0]{1}&A\\
     &\xrightarrow[1]{1}&-B\\
     &\xrightarrow[1]{0}&B\\
     &\xrightarrow[1]{0}&-B   
\end{eqnarray*}
  
  Since the result is $-B$, $\clf(2636,1143)=-1.$ Whenever the result is $-A$ or $-B$, $\clf =-1$ and whenever the result is $A$ or $B$, $\clf=+1.$
Since $101001001100 \text{\ XOR\ } 010001110111 = 111000111011 =  3643$ the result is
\begin{equation*}
  i_{2636}\cdot i_{1143} = -i_{3643}
\end{equation*}
or
\begin{equation*}
  e_{347ac}\cdot e_{123567b}=-e_{12456abc}
\end{equation*}

\end{enumerate}

\end{document}